\documentclass[scheme=plain]{ctexart}

\usepackage[a4paper,margin=1in]{geometry}
\usepackage{amsmath,amssymb,amsthm}
\usepackage{mathtools}
\usepackage{graphicx}
\IfFileExists{arxiv/references.bib}{}{}

\usepackage{float}
\usepackage{xurl}
\usepackage{hyperref}
\usepackage{tikz}
\usetikzlibrary{calc,arrows.meta,positioning,bending}
\usepackage{authblk}
\usepackage{pgfplots}
\pgfplotsset{compat=1.18}
\pgfplotsset{
  colormap={viridis}{
    rgb255=(68,1,84) rgb255=(72,40,120) rgb255=(62,74,137)
    rgb255=(49,104,142) rgb255=(38,130,142) rgb255=(31,158,137)
    rgb255=(53,183,121) rgb255=(110,206,88) rgb255=(181,222,43)
    rgb255=(253,231,37)
  }
}
\definecolor{crimson}{RGB}{220,20,60}
\definecolor{markerorange}{RGB}{230,140,10}

\newenvironment{restated}[1]{%
  \par\noindent\textbf{Theorem #1.}\itshape\ }%
  {\par}

\newenvironment{restatedconjecture}[1]{%
  \par\noindent\textbf{Conjecture #1.}\itshape\ }%
  {\par}

\newtheorem{theorem}{Theorem}[section]
\newtheorem{lemma}[theorem]{Lemma}
\newtheorem{definition}[theorem]{Definition}
\newtheorem{example}[theorem]{Example}
\newtheorem{remark}[theorem]{Remark}
\newtheorem{corollary}[theorem]{Corollary}

\newtheorem{proposition}[theorem]{Proposition}
\newtheorem{conjecture}[theorem]{Conjecture}

\newcommand{\dimbox}{\dim_\mathrm{B}}

\newcommand{\dimhaus}{\dim_\mathrm{H}}

\definecolor{moireblue}{HTML}{7DB7FF}
\definecolor{moirered}{HTML}{FF9E9E}

\newcommand{\SierTriPath}[4]{%
    \ifnum#1=0
        #2 -- #3 -- #4 -- cycle
    \else
        \SierTriPath{\number\numexpr#1-1\relax}{#2}{($ #2!0.5!#3 $)}{($ #4!0.5!#2 $)}%
        \SierTriPath{\number\numexpr#1-1\relax}{($ #2!0.5!#3 $)}{#3}{($ #3!0.5!#4 $)}%
        \SierTriPath{\number\numexpr#1-1\relax}{($ #4!0.5!#2 $)}{($ #3!0.5!#4 $)}{#4}%
    \fi
}

\newcommand{\RotatedMoireIntersection}[2]{%
    \begin{scope}
        \clip
            \SierTriPath{#1}{(0,1)}{(-0.8660254,-0.5)}{(0.8660254,-0.5)};
        \begin{scope}[rotate around={#2:(0,0)}]
            \path[fill=black]
                \SierTriPath{#1}{(0,1)}{(-0.8660254,-0.5)}{(0.8660254,-0.5)};
        \end{scope}
    \end{scope}
}

\newcommand{\RotatedMoireStage}[2]{%
    \path[fill=moireblue]
        \SierTriPath{#1}{(0,1)}{(-0.8660254,-0.5)}{(0.8660254,-0.5)};
    \begin{scope}[rotate around={#2:(0,0)}]
        \path[fill=moirered]
            \SierTriPath{#1}{(0,1)}{(-0.8660254,-0.5)}{(0.8660254,-0.5)};
    \end{scope}
    \RotatedMoireIntersection{#1}{#2}%
}

\newcommand{\LineMoirePattern}{%
    \def\A{1.50}%
    \begin{scope}[rotate around={-4:(0,0)}]
        \clip (-\A,-\A) rectangle (\A,\A);
        \foreach \y in {-3,-2.88,...,3} {
            \draw[line width=0.85pt, black, opacity=0.88] (-3.1,\y) -- (3.1,\y);
        }
    \end{scope}
    \begin{scope}[rotate around={4:(0,0)}]
        \clip (-\A,-\A) rectangle (\A,\A);
        \foreach \y in {-3,-2.88,...,3} {
            \draw[line width=0.85pt, black, opacity=0.52] (-3.1,\y) -- (3.1,\y);
        }
    \end{scope}
}

\title{Moir\'e Pattern of Rotated Sierpi\'nski Gaskets}
\author{Ziyu Neroli}
\affil{Department of Mathematics, Imperial College London, United Kingdom}

\date{}

\begin{document}

\maketitle

\begin{abstract}
    We study the moir\'e pattern formed by the intersection of a planar Sierpi\'nski gasket and a copy rotated about its centre.
    At every resonant angle, we prove that the intersection is of finite type and compute its Hausdorff and Minkowski dimensions from an explicit finite matrix; for Lebesgue-almost every angle, we bound its upper Minkowski dimension by $\overline{\dim}_{\mathrm B}(\mathcal S_\theta)\le2\dimhaus(S)-2$.
    Finally, we conjecture that both dimensions equal this bound at every non-resonant angle.
    Sections~2 and~3 are formalised in Lean.
\end{abstract}

\begin{center}
\begin{tikzpicture}[line join=round,line cap=round,scale=2.3]
    \RotatedMoireStage{6}{39}
    \begin{scope}[xshift=2.7cm]
        \RotatedMoireIntersection{6}{39}
    \end{scope}
\end{tikzpicture}
\end{center}

\tableofcontents

\section{Introduction}

\begin{figure}[ht]
\centering
\begin{minipage}[c]{0.29\textwidth}
\centering
\begin{tikzpicture}[line cap=round,line join=round,scale=0.72]
    \LineMoirePattern
\end{tikzpicture}
\par\smallskip
{\small (a) Rotated parallel-line layers}
\end{minipage}\hfill
\begin{minipage}[c]{0.36\textwidth}
\centering
\includegraphics[width=\linewidth]{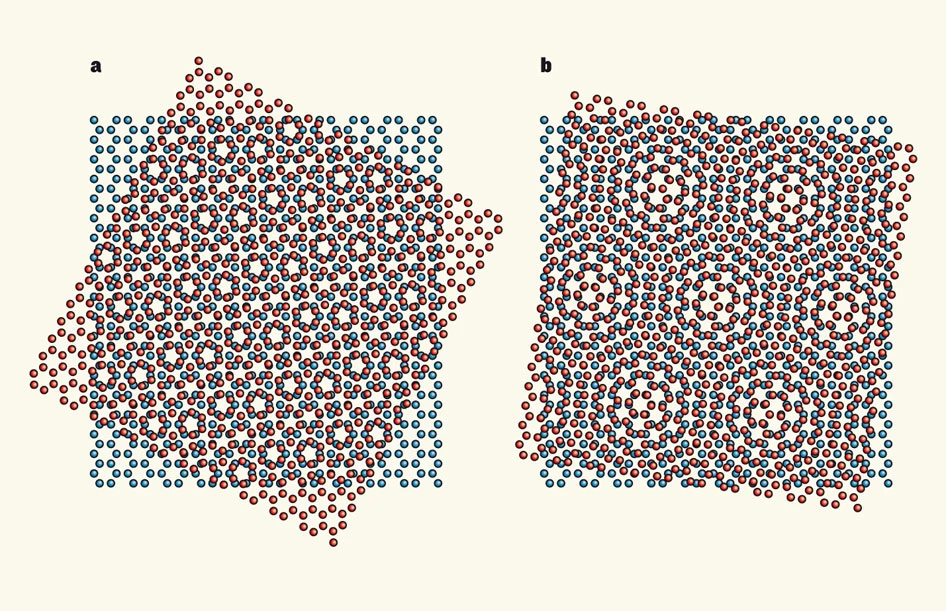}
\par\smallskip
{\small (b) Misoriented honeycomb lattices \cite{macdonald2011graphene}}
\end{minipage}\hfill
\begin{minipage}[c]{0.29\textwidth}
\centering
\begin{tikzpicture}[line join=round,line cap=round,scale=1.15]
    \RotatedMoireStage{4}{45}
\end{tikzpicture}
\par\smallskip
{\small (c) Moir\'e pattern $S\cap\mathbf R_{\pi/4}S$}
\end{minipage}
\caption{Three examples of moir\'e patterns. Panel (a) shows the classical visual effect obtained by superimposing two slightly misoriented families of parallel lines. Panel (b) shows two examples of misoriented honeycomb lattices from MacDonald and Bistritzer \cite{macdonald2011graphene}. Panel (c) shows the fourth pre-gasket approximation to a fractal moir\'e pattern $S\cap\mathbf R_\theta S$ at a representative non-resonant angle.}
\label{fig:moire-background-comparison}
\end{figure}

Moir\'e patterns originate from textile terminology: the word refers to the watered appearance of certain silk fabrics, and was later used for the large-scale geometric patterns produced by superimposing two similar line or lattice structures \cite{BritannicaMoire}.
The first scientific analysis of moir\'e fringes is usually traced to Rayleigh's 1874 work on diffraction gratings, where superposed gratings were used as a test for grating fidelity \cite{Rayleigh1874,Oster1965Moire}.
In modern materials science, moir\'e patterns are standard objects in layered two-dimensional materials: a small twist angle or a lattice mismatch between two crystal layers creates a long-wavelength spatial modulation, often called a moir\'e superlattice \cite{BistritzerMacDonald2011,Andrei2021Marvels,He2021MoireReview}.
Structural distortions can also modify electronic properties: Wu and Ganose showed that tilting in phosphide antiperovskites induces charge localisation and changes their photovoltaic properties \cite{WuGanose2024}.

Our motivation comes from the experimental realisation of molecular Sierpi\'nski triangle fractals.
Shang et al. assembled defect-free molecular Sierpi\'nski triangles on an Ag(111) surface by molecular self-assembly \cite{Shang2015Sierpinski}.
This suggests a natural geometric question: if a molecular layer has the geometry of a Sierpi\'nski gasket, what should its moir\'e pattern with a rotated copy look like?

\begin{figure}[h]
    \centering
    \includegraphics[width=0.8\linewidth]{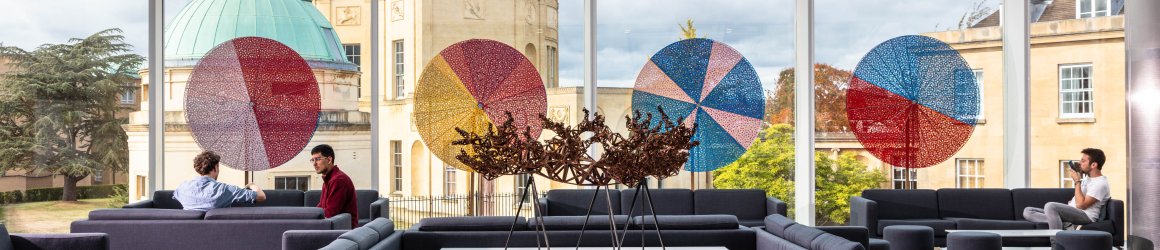}
    \caption{\emph{Beacons} in the common room of the Mathematical Institute, University of Oxford.}
    \label{fig:oxford-beacons}
\end{figure}

It is perhaps worth remarking that the generation of moiré patterns by rotation belongs to art as much as to mathematics. 
Four such pieces, Conrad Shawcross's \emph{Beacons}, hang in the common room of the Mathematical Institute at the University of Oxford (Figure~\ref{fig:oxford-beacons})\footnote{\url{https://www.maths.ox.ac.uk/about-us/art-and-oxford-mathematics/cascading-principles-expansions-within-geometry-philosophy-and-interference}}. 
Each is built from two coloured discs carrying a non-repeating pattern of perforations, rotating slowly in opposite directions, so that the shifting overlap of the two hole-patterns casts a continuously evolving moiré in the transmitted light. 
These works already move beyond the periodic moiré of classical optics towards something less regular; the rotational intersection of a Sierpi\'nski gasket with a rotated copy considered here may be read as a further step along the same line, in which the underlying pattern is self-similar and the dimension of its intersection with a rotated copy becomes the central question.
We note this only to suggest that the construction may lend itself to kinetic or optical art as readily as to dimension theory.

In this paper we isolate the geometric core of this question.
We use the term moir\'e pattern in a precise geometric sense.
Let $S$ be the standard planar Sierpi\'nski gasket, centred at the origin, and let $\mathbf R_\theta$ be rotation by angle $\theta$ about the origin.
We define the rotated moir\'e pattern at angle $\theta$ by
\[
\mathcal S_\theta:=S\cap \mathbf R_\theta S.
\]
Thus the centre is fixed, and only one copy of the gasket is rotated.
The long-term problem is to determine the Hausdorff dimension of $\mathcal S_\theta$ for every rotation angle $\theta$.
By the symmetries of the gasket, it is enough to consider $\theta\in[0,\pi/3]$.
We call $\theta$ a \emph{resonant angle} if $e^{i\theta}\in\mathbb Q(\omega)$, where $\omega:=e^{2\pi i/3}$, and non-resonant otherwise.


Related work on gasket intersections includes graph-directed constructions, translated intersections and linear slices.
McClure constructs finite graph-directed systems for intersections of self-similar sets with transformed copies under a finite-closure hypothesis, and gives a lattice criterion ensuring this hypothesis \cite[Theorem~1 and Lemma~2]{McClure2008}.
The finite intersection graph at our resonant angles is a specialisation of this construction; we group relative displacements into set-valued states to verify the open set condition and obtain the dimension formula.
McClure does not establish the converse classification of finite-type angles or the almost-everywhere upper bound proved here.
Cai and Li study translated intersections of generalised Sierpi\'nski gaskets with contraction ratio $1/q$, where $2<q<3$, under a unique-expansion condition on the translation coordinates \cite{CaiLi2020}; these hypotheses exclude the standard gasket considered here.
B\'ar\'any, Ferguson and Simon study dimensions and multifractal properties of planar gasket slices \cite{BaranyFergusonSimon2012}, while Nakajima treats fixed-direction hyperplane slices of higher-dimensional gaskets \cite{Nakajima2022}.
Pradhan et al. numerically investigate overlap statistics of regular and random gaskets in a fractal contact model \cite{PradhanEtAl2003}, and Motyka studies moir\'e patterns from generalised Cantor gratings \cite{Motyka2021}.

We study how the intersection dimension varies with the relative rotation, keeping the centres fixed. Our approach uses finite transition matrices at resonant angles and estimates for the difference measure at almost every non-resonant angle.

\subsection{Main results and conjecture}

In Section~\ref{sec:carry-geometry}, we determine the Hausdorff and box dimensions of the moir\'e pattern at every resonant angle.
The calculation uses an explicit finite matrix $\mathbf A_\theta$, whose entries count labelled transitions between sets of relative displacements; see Definition~\ref{def:relative-displacement-matrix}.
Recall that $\omega:=e^{2\pi i/3}$ and $\dimhaus(S)=\log3/\log2$.

\begin{restated}{\ref{thm:commensurable-gives-finite-type}}
Let $\theta\in[0,\pi/3]$. If $e^{i\theta}\in\mathbb Q(\omega)$, then
\[
\dimhaus(\mathcal S_\theta)=\dimbox(\mathcal S_\theta)
=\frac{\log\rho(\mathbf A_\theta)}{\log2}.
\]
\end{restated}

Theorem~\ref{thm:finite-type-implies-commensurable} proves the converse structural statement: a nonempty intersection with only finitely many reachable live relative displacements must occur at a resonant angle.
Thus the finite-type angles are exactly the resonant angles, as stated in Corollary~\ref{cor:finite-type-commensurable-dichotomy}.

In Section~\ref{sec:ae-upper}, we turn to rotations beyond this countable family.
We estimate the number of intersecting pairs of small gasket pieces through the planar difference measure and obtain the following upper bound.

\begin{restated}{\ref{thm:ae-upper}}
For Lebesgue-almost every $\theta$,
\[
\overline{\dim}_{\mathrm B}(\mathcal S_\theta)\le 2\dimhaus(S)-2.
\]
\end{restated}

The remaining problem is to determine both dimensions at every non-resonant angle.
Even the matching Hausdorff lower bound for almost every angle is open: our two gaskets have fixed coincident centres, so a result for generic translations does not settle this intersection.
Section~\ref{sec:ae-lower} discusses the obstacles to the lower bound.

Figure~\ref{fig:intro-resonant-dimensions} shows dimensions computed from the spectral-radius formula at resonant angles, together with finite-scale estimates at selected non-resonant angles.
The non-resonant estimates are close to $2\dimhaus(S)-2$, providing numerical support for the conjecture below.
Section~\ref{sec:conjecture} describes the computations and their limitations.

\begin{figure}[H]
\centering
\includegraphics[width=0.7\linewidth]{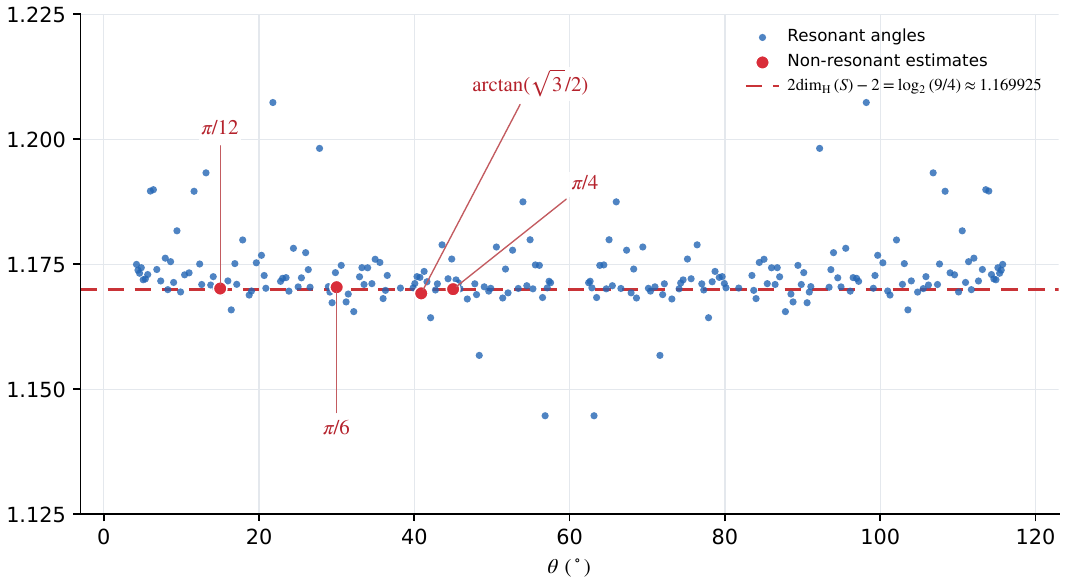}
\caption{Dimensions and finite-scale estimates.}
\label{fig:intro-resonant-dimensions}
\end{figure}

\begin{restatedconjecture}{\ref{conj:non-resonant-dimension}}
Let $\theta\in[0,\pi/3]$ with $e^{i\theta}\notin\mathbb Q(\omega)$.
Then
\[
\dimhaus(\mathcal S_\theta)=\dimbox(\mathcal S_\theta)=2\dimhaus(S)-2.
\]
\end{restatedconjecture}

\subsection{Lean formalisation}
Generally Sections~\ref{sec:carry-geometry} and~\ref{sec:ae-upper} have been formalised in Lean.

The results of Section~\ref{sec:carry-geometry} are formalised in Lean for the gasket defined by infinite address sums.
The development verifies the displacement recursion, the resonant-angle characterisations, the finite-state graph and its Hausdorff and box dimension formula, including the example at $\theta=\pi/3$.
The dimension proof constructs a path probability measure and proves its geometric mass estimates; it does not assume the graph-directed dimension theorem.
Lean also verifies the polygon-filling identities in Lemma~\ref{lem:filling}.
For Section~\ref{sec:ae-upper}, Lean verifies the actual measures and their self-similar identities, both almost-everywhere separation lemmas and the directional-resonance characterisation.
The proof in Section~\ref{sec:ae-upper} uses five analytic results from two papers as black boxes.
Three come from Corso--Shmerkin: the planar ambient-dimension bound \cite[Section~1.2]{corso2024dynamical} and the one-dimensional and planar specialisations of \cite[Corollary~4.2]{corso2024dynamical}.
The other two are the one-dimensional ambient-dimension bound and the convolution inequality from Rossi--Shmerkin \cite[Section~1, equation~(1.2)]{rossi2018measures}.
These five results are declared as external axioms; their proofs are not formalised here.
Lean verifies all their hypotheses for the measures used in this paper, including compact support, self-similarity, exponential separation and the required projection inequalities, and checks the remaining arguments.
Section~\ref{sec:carry-geometry} is independent of these black boxes.

The development uses Lean and mathlib version 4.32.1.
The source, dependency manifest, build instructions and verification records are stored at
\begin{center}
\url{https://github.com/Nero-17/LEAN-Formalisation-Moire-Pattern-Sierpinski-Gasket}.
\end{center}
The repository is public and contains the Lean development, numerical data and English documentation; the manuscript is not hosted there.
The accompanying dependency audit records the precise sources and specialisations of the five external axioms and checks that no other custom axioms or unfinished proofs enter the final results.

\section{Resonant angles and finite-type intersections}
\label{sec:carry-geometry}

The self-similarity of the gasket turns an intersection at one scale into nine
intersection problems at the next.  We first describe this construction for an
arbitrary rotation and translation; the arithmetic of the rotation will enter
only in Subsection~\ref{sec:commensurable}.

Identify the plane with $\mathbb C$, put $\omega:=e^{2\pi i/3}$, and take the
vertices $\vec p_0:=1$, $\vec p_1:=\omega$ and $\vec p_2:=\omega^2$.
Thus the triangle $T:=\operatorname{conv}\{\vec p_0,\vec p_1,\vec p_2\}$ has
centre $\vec{0}$.  For $\mathcal I:=\{0,1,2\}$, the gasket $S$ is the unique
non-empty compact set satisfying
\[
S=\bigcup_{i\in\mathcal I}F_i(S),
\qquad F_i(\mathbf x):=\frac{\mathbf x+\vec p_i}{2}.
\]
Each $F_i(S)$ is a half-size copy of $S$, centred at $\vec p_i/2$.
Here and below, the centre of a gasket means the centre of its convex hull.
For a word $\mathbf i:=(i_1,\ldots,i_n)\in\mathcal I^n$, define $F_{\mathbf i}:=F_{i_1}\circ\cdots\circ F_{i_n}$.

Let $\mathbf R_\theta$ denote rotation about the origin and write
$S_\theta:=\mathbf R_\theta S$.
For a relative displacement $\vec c\in\mathbb R^2$, define
$E_{\vec c}^{(\theta)}:=S\cap(S_\theta-\vec c)$, so that
$E_{\vec{0}}^{(\theta)}=\mathcal S_\theta$.
We draw the fixed copy in blue and the rotated copy in red.
With this sign convention, $\vec c$ points from the red centre $-\vec c$ to
the blue centre $\vec{0}$.
The following three steps show how this relative displacement changes when we select
one child from each gasket and restore them to unit size.

\paragraph{Step 1: split into nine pairs.}
For any $\theta$ and $\vec c$, decompose both gaskets into their three
first-level pieces and distribute the intersection over the two unions:
\[
E_{\vec c}^{(\theta)}
=\bigcup_{i,j\in\mathcal I}
\bigl[F_i(S)\cap(\mathbf R_\theta F_j(S)-\vec c)\bigr].
\]
Some of these nine intersections may be empty, and distinct pairs may overlap.
Figure~\ref{fig:carry-nine-pairs} displays all nine choices; we follow the
highlighted pair through the next two steps.

\begin{figure}[H]
\centering
\includegraphics[width=\linewidth]{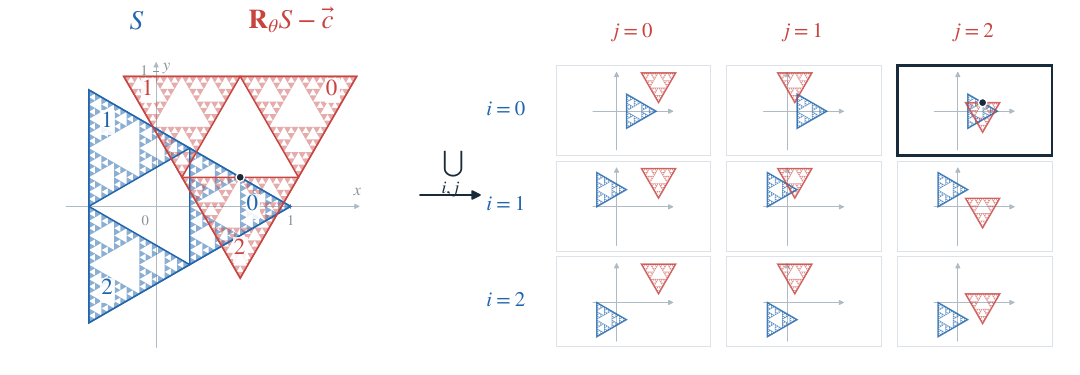}
\caption{The nine first-level pairs.  Rows select the blue piece and columns the
red piece; the marked pair is $(i,j)=(0,2)$.
All three diagrams use $\theta=\pi/6$ and
$\vec c=(-5/8,-(2+\sqrt3)/8)$, with finite approximations to the gaskets.}
\label{fig:carry-nine-pairs}
\end{figure}

\paragraph{Step 2: restore the scale, keeping the angle.}
Fix $i,j\in\mathcal I$ and apply the same affine map
$\mathbf z\mapsto2\mathbf z-\vec p_i$ to both selected pieces.
The blue piece becomes $S$; the red piece becomes a unit-size gasket with
the same rotation angle $\theta$.
Denote its new relative displacement by $\vec c'_{ij}$, so that it is
$\mathbf R_\theta S-\vec c'_{ij}$.
The map is a bijection with inverse $F_i$, hence the original pairwise
intersection is exactly $F_i(E_{\vec c'_{ij}}^{(\theta)})$.

\begin{figure}[H]
\centering
\includegraphics[width=\linewidth]{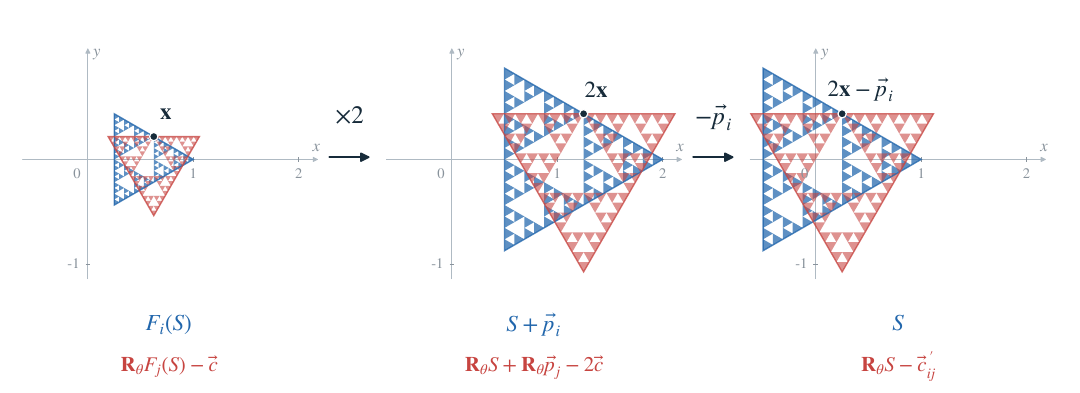}
\caption{A common dilation followed by a common translation.  The black point
tracks the same intersection point, and all panels have the same scale.}
\label{fig:carry-normalisation}
\end{figure}

\paragraph{Step 3: read off the new relative displacement.}
Before rescaling, the selected blue and red centres are $\vec p_i/2$ and
$\mathbf R_\theta\vec p_j/2-\vec c$.
The relative displacement from red to blue can be read along the three arrows in
Figure~\ref{fig:carry-displacement}: from the red child back to its parent,
then to the blue parent, then to the blue child.
These arrows contribute $-\mathbf R_\theta\vec p_j/2$, $\vec c$ and
$\vec p_i/2$, respectively.
The normalisation doubles every relative displacement, giving
\[
\vec c'_{ij}
=2\left(\vec c+\frac{\vec p_i-\mathbf R_\theta\vec p_j}{2}\right)
=2\vec c+\vec p_i-\mathbf R_\theta\vec p_j.
\]

\begin{figure}[H]
\centering
\includegraphics[width=\linewidth]{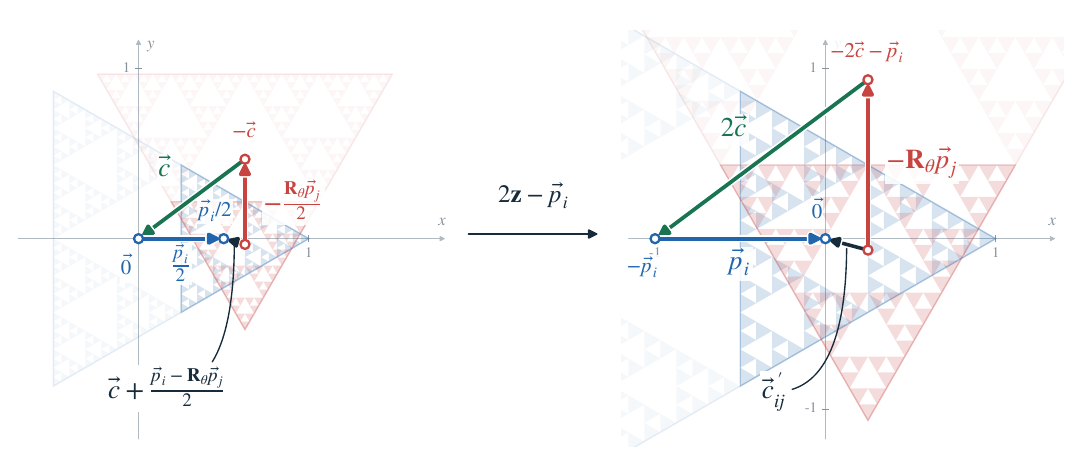}
\caption{The three arrows and their resultant, before and after normalisation.}
\label{fig:carry-displacement}
\end{figure}

As a result, we have the following statement.

\begin{lemma}
\label{lem:carry-recursion-general}
For every $\theta\in\mathbb R$ and $\vec c\in\mathbb R^2$,
\[
E_{\vec c}^{(\theta)}
=\bigcup_{i,j\in\mathcal I}
F_i\Bigl(E_{2\vec c+\vec p_i-\mathbf R_\theta\vec p_j}^{(\theta)}\Bigr).
\]
\end{lemma}
\begin{proof}
    By the three steps.
\end{proof}

\subsection{Resonant angles}
\label{sec:commensurable}

Up to a common rigid motion, the intersection of two copies of $S$ is
determined by their relative rotation $\theta$ and relative displacement $\vec c$.
The preceding three steps show that each pair of child gaskets, after a common
normalisation, gives an intersection of the same form, with the same relative
rotation $\theta$ and a new relative displacement.
Thus, for fixed $\theta$, we can recover the intersection by following the
normalised relative displacements for all pairs of children and mapping the
resulting intersections back to the original pieces.

For $\mathcal S_\theta$, the two initial centres coincide, so put
$\vec c_0:=\vec{0}$.
If $i_{n+1},j_{n+1}\in\mathcal I$ select the two children at the next level,
Lemma~\ref{lem:carry-recursion-general} gives
\begin{equation}
\vec c_{n+1}
=2\vec c_n+\vec p_{i_{n+1}}-\mathbf R_\theta\vec p_{j_{n+1}},
\qquad n\geq0.
\label{eq:carry-update}
\end{equation}
Iterating this relation yields
\begin{equation}
\vec c_n=\sum_{r=1}^n2^{n-r}
\bigl(\vec p_{i_r}-\mathbf R_\theta\vec p_{j_r}\bigr),
\qquad n\geq1.
\label{eq:carry-closed-form}
\end{equation}
Here $\vec c_n$ is the relative displacement after normalisation, while
$2^{-n}\vec c_n$ is the relative displacement from the red reference centre
to the blue reference centre at the original scale.
A relative displacement $\vec c$ is \emph{reachable} if it equals $\vec c_n$
for some finite pair of address prefixes, and is \emph{live} if
$E_{\vec c}^{(\theta)}\neq\varnothing$.
We say that $\theta$ is \emph{of finite type} if only finitely many live
relative displacements are reachable from $\vec{0}$.

We can decide whether a relative displacement is live by testing membership in a convex polygon.

\begin{lemma}
\label{lem:filling}
For every $\theta$,
\[
S-\mathbf R_\theta S=T-\mathbf R_\theta T,
\qquad
\mathbf R_\theta S-S=\mathbf R_\theta T-T.
\]
\end{lemma}

\begin{proof}
The inclusion $S-\mathbf R_\theta S\subseteq T-\mathbf R_\theta T$ is immediate.
For $\vec z\in T-\mathbf R_\theta T$, the triangles $T$ and $\vec z+\mathbf R_\theta T$ intersect.
Their boundaries must meet: otherwise one triangle would lie strictly inside the other, contradicting equality of their areas.
Since $\partial T\subset S$, a common boundary point belongs to both $S$ and $\vec z+\mathbf R_\theta S$.
Thus $\vec z\in S-\mathbf R_\theta S$, proving the first equality; taking negatives gives the second.
The polygon $T-\mathbf R_\theta T$ is convex and contains $\vec{0}$ because both triangles contain their common centroid.
\end{proof}

Thus $E_{\vec c}^{(\theta)}\neq\varnothing$ if and only if $\vec c\in\mathbf R_\theta T-T$.
This closed polygon is the convex hull of the nine points $\mathbf R_\theta\vec p_j-\vec p_i$, with $i,j\in\mathcal I$; points on its boundary are included.
We use this test to construct the labelled graph.

\begin{definition}
\label{def:relative-displacement-matrix}
Fix $\theta\in[0,\pi/3]$ and start from the state $\{\vec{0}\}$.
For each state $\mathcal A$ and label $i\in\mathcal I$, form
\[
\mathcal B:=\bigl\{2\vec c+\vec p_i-\mathbf R_\theta\vec p_j:
\vec c\in\mathcal A,\ j\in\mathcal I\bigr\}
\cap(\mathbf R_\theta T-T).
\]
If $\mathcal B\neq\varnothing$, include the edge
$\mathcal A\xrightarrow{i}\mathcal B$, with similarity $F_i$.
Retain only states reachable from $\{\vec{0}\}$, identifying equal sets.
Define the adjacency matrix $\mathbf A_\theta$ by
\[
(\mathbf A_\theta)_{\mathcal A,\mathcal B}
:=\#\{i\in\mathcal I:\mathcal A\xrightarrow{i}\mathcal B\}.
\]
Thus each state is a finite set of relative displacements, and each label gives
at most one outgoing edge. If $\theta$ is of finite type, this matrix is finite.
\end{definition}

To construct $\mathbf A_\theta$, keep a queue initially containing only $\{\vec{0}\}$.
Remove a state $\mathcal A$ from the queue and, for each label $i$, form $\mathcal B$ by the polygon test in the definition.
Discard an empty $\mathcal B$; otherwise record the edge $\mathcal A\xrightarrow{i}\mathcal B$ and add $\mathcal B$ to the queue if that set has not been encountered before.
Process each state once, stopping when the queue is empty, and count the labels between each pair of states.
States are identified by equality of their displacement sets, not by equality of the intersections they represent.

At a resonant angle, this procedure terminates.
Clearing denominators gives a positive
integer $q$ with $\mathbf R_\theta\boldsymbol\Lambda\subset q^{-1}\boldsymbol\Lambda$, where the vertices generate the Eisenstein lattice $\boldsymbol{\Lambda}:=\mathbb Z\vec p_0+\mathbb Z\vec p_1=\mathbb Z[\omega]$.
The recursion then places every reachable relative displacement in
$q^{-1}\boldsymbol\Lambda$. A live relative displacement also belongs to the
bounded set $\mathbf R_\theta T-T$, so there are only finitely many such
displacements. The states in Definition~\ref{def:relative-displacement-matrix}
are subsets of this finite set, so the queue contains only finitely many distinct states.
Given the rational coordinates of $e^{i\theta}$ in the basis $(1,\omega)$, every step can be performed exactly: in the basis $(\vec p_0,\vec p_1)$, the rotation matrix, polygon vertices and candidate displacements all have rational coordinates.
Polygon membership and equality of states therefore require only rational arithmetic.

\begin{example}
\label{ex:pi-three-matrix}
At $\theta=\pi/3$, Figure~\ref{fig:pi-three-states} shows the intersection represented by each reachable state.
The black corner labels give the intersections after applying $F_i^{-1}$;
the table records their relative displacements, so $\{-\vec p_0,-\vec p_2\}$ represents $E_{-\vec p_0}^{(\pi/3)}\cup E_{-\vec p_2}^{(\pi/3)}$.

\begin{center}
\includegraphics[width=.96\linewidth]{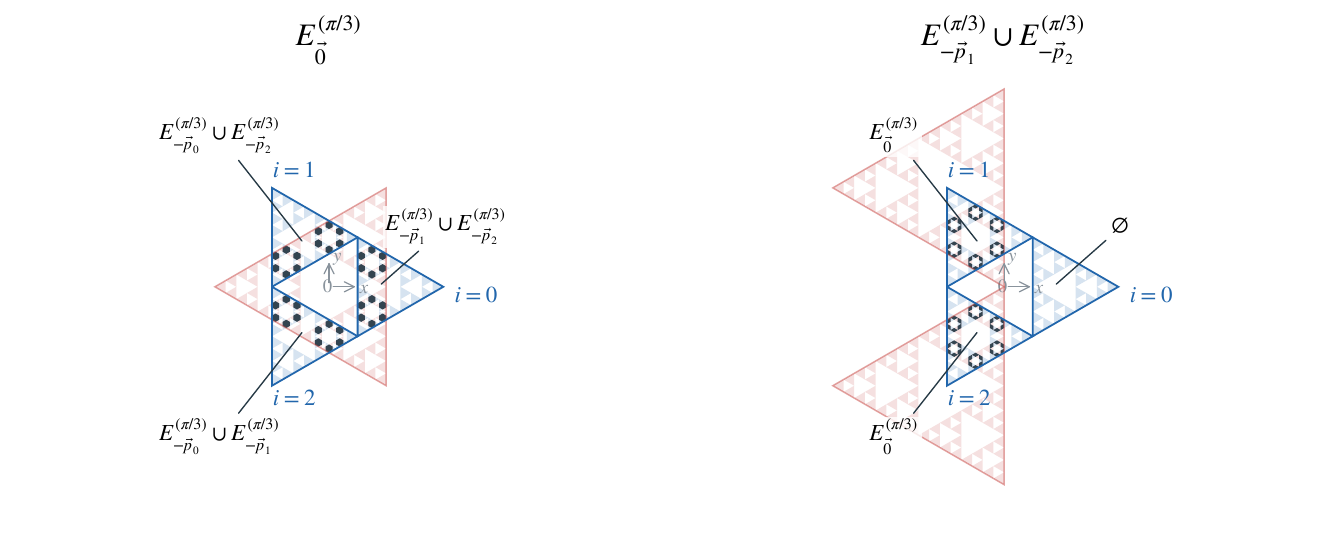}
\end{center}

\begin{figure}[H]
\centering
\includegraphics[width=.96\linewidth]{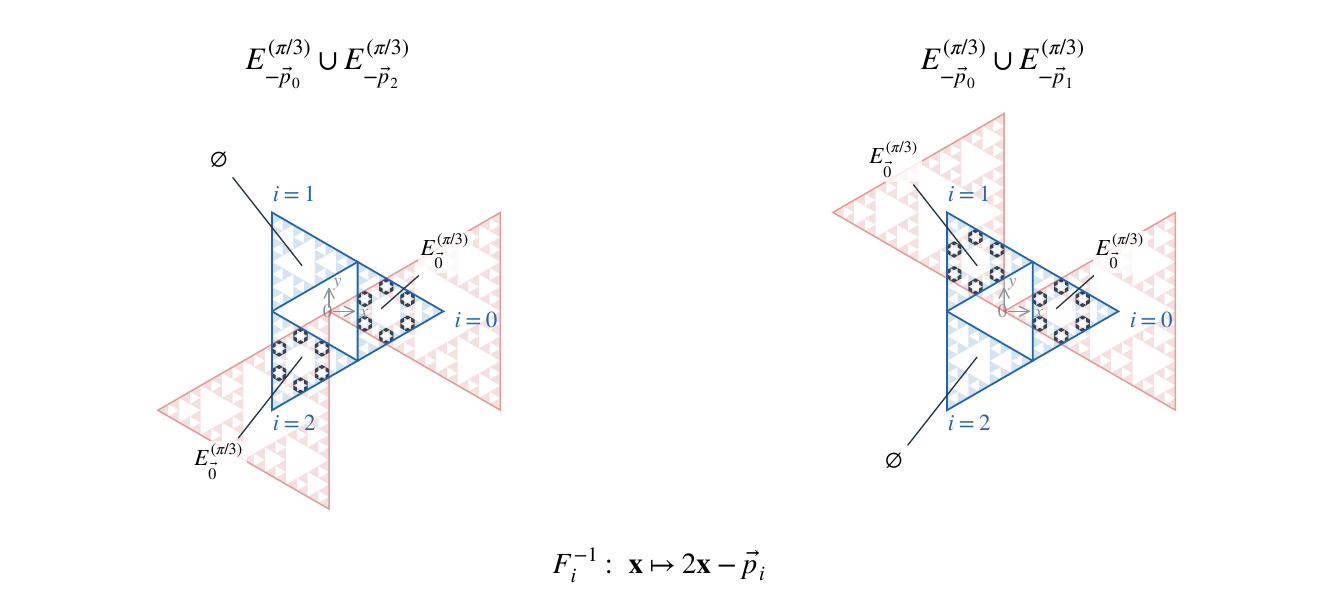}
\caption{The intersections represented by the four states at $\theta=\pi/3$, in row order in the table below.
Blue shows $S$; the full red gaskets are the corresponding copies $\mathbf R_{\pi/3}S-\vec c$, and their intersections with $S$ are dark.
Fourth-level approximations are drawn, with identical coordinates in all panels.
The blue index $i$ selects a corner; its black label identifies the target intersection.}
\label{fig:pi-three-states}
\end{figure}

Within $\mathbf R_{\pi/3}T-T=-2T$, the reachable relative displacements are
$\vec{0},-\vec p_0,-\vec p_1,-\vec p_2$.
Each lies on a directed cycle, so all four are live.
In the table, rows are source states $\mathcal A$, columns are target states $\mathcal B$,
and each cell lists the labels $i$ for which $\mathcal A\xrightarrow{i}\mathcal B$.
Counting the labels in each cell gives the corresponding matrix entry:
\begingroup
\setlength{\arraycolsep}{3pt}
\[
\begin{array}{c|cccc}
\mathcal A\backslash\mathcal B & \{\vec{0}\} & \{-\vec p_1,-\vec p_2\} & \{-\vec p_0,-\vec p_2\} & \{-\vec p_0,-\vec p_1\}\\
\hline
\{\vec{0}\} & \varnothing & \{0\} & \{1\} & \{2\}\\
\{-\vec p_1,-\vec p_2\} & \{1,2\} & \varnothing & \varnothing & \varnothing\\
\{-\vec p_0,-\vec p_2\} & \{0,2\} & \varnothing & \varnothing & \varnothing\\
\{-\vec p_0,-\vec p_1\} & \{0,1\} & \varnothing & \varnothing & \varnothing
\end{array}
\quad\Longrightarrow\quad
\mathbf A_{\pi/3}=
\begin{pmatrix}
0&1&1&1\\
2&0&0&0\\
2&0&0&0\\
2&0&0&0
\end{pmatrix}.
\]
\endgroup
Here $\varnothing$ means there is no edge.
The positive eigenvector $(3,\sqrt6,\sqrt6,\sqrt6)^{\mathrm T}$ gives
$\rho(\mathbf A_{\pi/3})=\sqrt6$.
\end{example}

\begin{theorem}
\label{thm:commensurable-gives-finite-type}
Let $\theta\in[0,\pi/3]$. If
$
e^{i\theta}\in\mathbb Q(\omega),
$
then
\[
\dimhaus(\mathcal S_\theta)=\dimbox(\mathcal S_\theta)
=\frac{\log\rho(\mathbf A_\theta)}{\log2}.
\]
\end{theorem}

\begin{proof}
The construction following Definition~\ref{def:relative-displacement-matrix} shows that $\mathbf A_\theta$ is finite.

Each state $\mathcal A$ represents
$E_{\mathcal A}^{(\theta)}:=\bigcup_{\vec c\in\mathcal A}E_{\vec c}^{(\theta)}$.
Lemma~\ref{lem:carry-recursion-general} gives its graph-directed decomposition,
with the initial state representing $\mathcal S_\theta$.
These sets are nonempty: the initial gaskets contain the intersecting boundaries
of their centred congruent triangles, and subsequent states contain only live
relative displacements.
Use $\operatorname{int}T$ as the open set at every state. Distinct labels give
disjoint corners $F_i(\operatorname{int}T)\subset\operatorname{int}T$, and each
label occurs at most once at a state. Thus the graph-directed open set condition
holds, and every edge contracts by $1/2$.
All states are reachable from the initial state, so the graph-directed dimension
theorem \cite{MauldinWilliams1988} gives equality of its Hausdorff and box
dimensions, with common value determined by
$\rho(2^{-s}\mathbf A_\theta)=1$.
Since $\rho(2^{-s}\mathbf A_\theta)=2^{-s}\rho(\mathbf A_\theta)$, this value is
$\log\rho(\mathbf A_\theta)/\log2$.
\end{proof}

\begin{example}
\label{thm:pi-three-dimension}
Since $e^{i\pi/3}=1+\omega\in\mathbb Q(\omega)$ and $\rho(\mathbf A_{\pi/3})=\sqrt6$, Theorem~\ref{thm:commensurable-gives-finite-type} gives 
\[
\dimhaus(\mathcal S_{\pi/3})=\dimbox(\mathcal S_{\pi/3})=\log6/\log4.
\]
\end{example}

\begin{figure}[t]
\centering
\includegraphics[width=\linewidth]{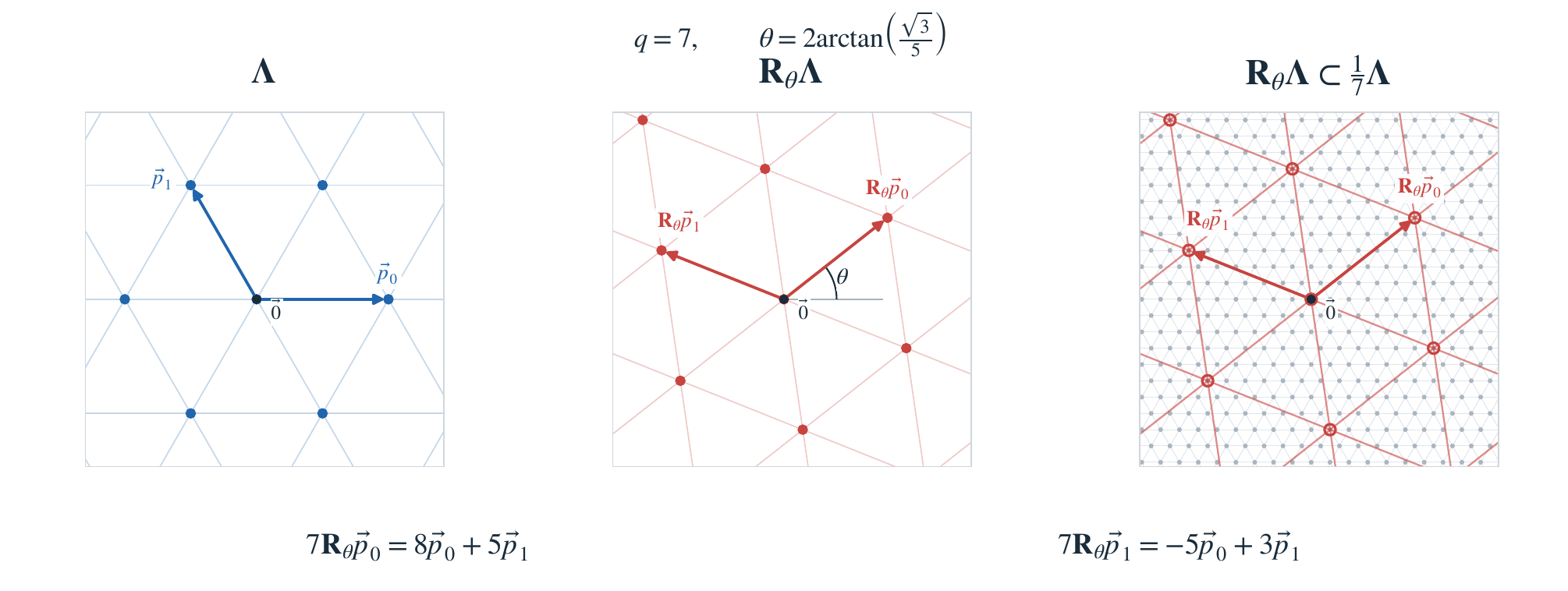}
\caption{A resonant angle with $q=7$ and $\theta=2\arctan(\sqrt3/5)$.
In the right panel, every point of the rotated lattice (red) lies in the finer lattice $7^{-1}\boldsymbol\Lambda$ (grey).
Lines connect neighbouring lattice points.}
\label{fig:commensurable-lattice-inclusion}
\end{figure}

The theorem's hypothesis has the following equivalent forms.

\begin{samepage}
\begin{proposition}
\label{prop:commensurable-equivalences}
Let $\theta\in[0,\pi/3]$. The following four conditions are equivalent:
\begin{enumerate}
\item $e^{i\theta}\in\mathbb Q(\omega)$.
\item $\dfrac1{\sqrt3}\tan\dfrac\theta2\in\mathbb Q$.
\item There exist coprime integers $a,b$ with $a\ge0$ and $0\le2b\le a$
such that $\theta=2\arctan\dfrac{\sqrt3\,b}{2a-b}$.
\item There exists a positive integer $q$ with
$\mathbf R_\theta\boldsymbol\Lambda\subset q^{-1}\boldsymbol\Lambda$
(Figure~\ref{fig:commensurable-lattice-inclusion}).
\end{enumerate}
\end{proposition}
\end{samepage}

\begin{proof}
\noindent\emph{(1) $\Rightarrow$ (2).}
Since $e^{i\theta}\in\mathbb Q(\omega)$, both $\cos\theta$ and $\sin\theta/\sqrt3$ are rational. Hence
\[
\frac1{\sqrt3}\tan\frac\theta2
=\frac{\sin\theta/\sqrt3}{1+\cos\theta}\in\mathbb Q.
\]

\noindent\emph{(2) $\Rightarrow$ (3).}
Write $\tfrac1{\sqrt3}\tan\tfrac\theta2=r/s$ in lowest terms, with $s>0$ and
$0\le3r\le s$. Let $(a,b)$ be the coprime pair obtained by dividing
$(s+r,2r)$ by its greatest common divisor. Then $a\ge0$, $0\le2b\le a$ and
$b/(2a-b)=r/s$, so $\theta=2\arctan\dfrac{\sqrt3\,b}{2a-b}$.

\noindent\emph{(3) $\Rightarrow$ (4).}
Here $\arg(a+b\omega)=\theta/2$. For the positive integer $q:=a^2-ab+b^2$,
\[
e^{i\theta}=\frac{a+b\omega}{a+b\overline\omega}
=\frac{(a+b\omega)^2}{q}.
\]
Multiplication by $(a+b\omega)^2$ maps $\boldsymbol\Lambda$ into itself,
so $\mathbf R_\theta\boldsymbol\Lambda\subset q^{-1}\boldsymbol\Lambda$.

\noindent\emph{(4) $\Rightarrow$ (1).}
Since $\vec p_0=1\in\boldsymbol\Lambda$,
$e^{i\theta}=\mathbf R_\theta\vec p_0\in q^{-1}\boldsymbol\Lambda\subset\mathbb Q(\omega)$.
\end{proof}

\subsection{Finite-type angles are exactly the resonant angles}

A natural question is whether a non-resonant angle can also yield a nonempty finite-type intersection.
We follow the normalised relative displacements along a surviving sequence of child intersections.
Finiteness forces a displacement to repeat, and this repetition makes the rotation map one non-zero lattice vector to another.

\begin{theorem}
\label{thm:finite-type-implies-commensurable}
Let $\theta\in[0,\pi/3]$.
Assume that the initial state $\vec{0}$ is live and that only finitely many live relative displacements are reachable from $\vec{0}$.
Then $e^{i\theta}\in\mathbb Q(\omega)$.
\end{theorem}

\begin{proof}
Choose a point of $E_{\vec{0}}^{(\theta)}$ and addresses $(i_r)_{r\ge1}$ and $(j_r)_{r\ge1}$ for that point in the two gaskets.
The corresponding prefixes give a live relative displacement $\vec c_n$ at every level.
Since only finitely many such displacements are reachable, $\vec c_m=\vec c_n$ for some $0\le m<n$.
Subtracting the two expressions in~\eqref{eq:carry-closed-form} gives
\[
\begin{aligned}
\mathbf R_\theta\left(
\sum_{r=1}^n2^{n-r}\vec p_{j_r}
-\sum_{r=1}^m2^{m-r}\vec p_{j_r}\right)
&=\sum_{r=1}^n2^{n-r}\vec p_{i_r}
-\sum_{r=1}^m2^{m-r}\vec p_{i_r}.
\end{aligned}
\]
Both differences are lattice vectors. We show that the one in parentheses is non-zero.

The vertices $\vec p_0,\vec p_1,\vec p_2$ represent distinct non-zero classes modulo $2\boldsymbol\Lambda$.
Thus equality of two finite address sums determines the same last digit on both sides; subtracting that digit and dividing by $2$ removes it.
Repeating this operation would reduce equality of the two red sums above to a nonempty address sum equal to $\vec{0}$, since $n>m$.
This is impossible: modulo $2\boldsymbol\Lambda$, any nonempty address sum equals its last vertex, which is non-zero.

Under the complex identification, $\mathbf R_\theta$ is multiplication by $e^{i\theta}$ and both lattice vectors belong to $\mathbb Z[\omega]$.
Dividing by the non-zero red difference therefore gives $e^{i\theta}\in\mathbb Q(\omega)$.
\end{proof}

Together with the forward finiteness argument, this gives the following characterisation.

\begin{corollary}
\label{cor:finite-type-commensurable-dichotomy}
An angle $\theta\in[0,\pi/3]$ is of finite type if and only if it is resonant.
These angles form a countable dense subset of $[0,\pi/3]$.
\end{corollary}

\begin{proof}
The boundaries of the centred congruent triangles $T$ and $\mathbf R_\theta T$ intersect, and $\partial T\subset S$, so the initial state is live for every $\theta$.
The equivalence follows from Theorem~\ref{thm:finite-type-implies-commensurable} and the finiteness argument following Definition~\ref{def:relative-displacement-matrix}.
By Proposition~\ref{prop:commensurable-equivalences}, the resonant angles are precisely the image of $\mathbb Q\cap[0,1/3]$ under the continuous, strictly increasing map $x\mapsto2\arctan(\sqrt3\,x)$.
This image is countable and dense in $[0,\pi/3]$.
\end{proof}

For these angles, Theorem~\ref{thm:commensurable-gives-finite-type} gives the dimension through the finite matrix $\mathbf A_\theta$.
We next turn to a bound that holds for almost every angle.

\section{An upper box bound for almost every angle}
\label{sec:ae-upper}

The finite-type result applies to a countable set of angles.
To treat almost every angle, we count intersecting pairs of small gasket cells through the distribution of the difference between two independent gasket points.
We first establish separation estimates for almost every angle, then use Corso--Shmerkin's $L^q$-dimension formula \cite[Corollary~4.2]{corso2024dynamical}, restated as Theorem~\ref{thm:corso-shmerkin} below, to bound the mass of this distribution near the origin.
This yields the upper bound $\overline{\dim}_{\mathrm B}(\mathcal S_\theta)\le2\dimhaus(S)-2$.

Throughout this section, $\theta\in[0,2\pi)$; periodicity gives the same almost-everywhere statements on $\mathbb R$.
Let $\mu$ be the uniform self-similar probability measure on $S$, so that $\mu=\tfrac13\sum_{i=0}^2(F_i)_*\mu$.

\begin{definition}
The \emph{difference measure} is $\nu_\theta:=\mu*((-\mathbf R_\theta)_*\mu)$, the law of $\vec X-\mathbf R_\theta\vec Y$ for independent $\vec X,\vec Y$ with law $\mu$.
\end{definition}

It is generated by the nine maps
\[
\phi_{ij}^{\theta}(\vec z):=\frac{\vec z+\vec p_i-\mathbf R_\theta\vec p_j}{2},
\qquad i,j\in\mathcal I,
\]
with equal weights $1/9$.
By Lemma~\ref{lem:filling}, the support of $\nu_\theta$ is the compact convex polygon $T-\mathbf R_\theta T$, which contains $\vec{0}$.

\subsection{Separation for almost every angle}

For a homogeneous iterated function system, \emph{exponential separation} means that some $a>0$ gives a lower bound $e^{-an}$ for the distance between cylinder translations corresponding to any two distinct length-$n$ address words, at infinitely many levels $n$.
We will obtain such a bound at every sufficiently large level for the difference system.
The estimates use the lattice $\boldsymbol\Lambda=\mathbb Z[\omega]$: every non-zero lattice vector has length at least $1$, and the number with lattice coordinates bounded by $C2^n$ is $O(4^n)$.

\begin{lemma}
\label{lem:ae-exp-sep}
For Lebesgue-almost every $\theta$, the system $\{\phi_{ij}^{\theta}:i,j\in\mathcal I\}$ satisfies exponential separation.
\end{lemma}

\begin{proof}
A difference between two level-$n$ cylinder translations has the form $2^{-n}(\vec u-\mathbf R_\theta\vec v)$, where
\[
\vec u:=\sum_{r=1}^n2^{n-r}(\vec p_{i_r}-\vec p_{i'_r}),
\qquad
\vec v:=\sum_{r=1}^n2^{n-r}(\vec p_{j_r}-\vec p_{j'_r}).
\]
These lattice vectors have coordinates bounded by $2^{n+1}$.
The modulo-$2\boldsymbol\Lambda$ argument in the proof of Theorem~\ref{thm:finite-type-implies-commensurable} shows that the address sums are injective.
Consequently, distinct word pairs give $(\vec u,\vec v)\neq(\vec{0},\vec{0})$.

For a fixed such pair and $0<\delta<1/2$, the set of angles satisfying $|\vec u-\mathbf R_\theta\vec v|<\delta$ has length at most $C\delta$, with an absolute constant $C$.
If either vector is zero, this set is empty.
Otherwise, $\mathbf R_\theta\vec v$ traces a circle of radius $|\vec v|\ge1$, and its intersection with the disk $B(\vec u,\delta)$ occupies an angular interval of length at most $2\arcsin(\delta/|\vec v|)\le C\delta$.

Let $B_n$ be the set of angles for which $|\vec u-\mathbf R_\theta\vec v|<16^{-n}n^{-2}$ for some non-zero pair with the level-$n$ coordinate bounds.
There are at most $C16^n$ such pairs, so $|B_n|\le Cn^{-2}$.
Since $\sum_n|B_n|<\infty$, the first Borel--Cantelli lemma implies that almost every angle belongs to only finitely many $B_n$; no independence is required.
For each remaining angle, the translations of any two distinct length-$n$ words are separated by at least
\[
2^{-n}16^{-n}n^{-2}=32^{-n}n^{-2}\ge e^{-4n}
\]
for all sufficiently large $n$.
\end{proof}

We also need separation in at least one factor of each projected convolution.
For $\vec e\in S^1$, write $\pi_{\vec e}(\vec x):=\langle\vec e,\vec x\rangle$ and $\mu_{\vec e}:=(\pi_{\vec e})_*\mu$.
Then $(\pi_{\vec e})_*\nu_\theta=\mu_{\vec e}*((-\mathrm{id})_*\mu_{\mathbf R_\theta^T\vec e})$.

We call $\theta$ a \emph{directional resonance} if $\mathbf R_\theta\vec\eta$ is parallel to $\vec\lambda$ for some non-zero $\vec\lambda,\vec\eta\in\boldsymbol\Lambda$.
Equivalently, $\tan\theta\in\sqrt3\,\mathbb Q\cup\{\infty\}$.
These angles form a countable set, denoted by $\Theta_0$.
This condition is weaker than resonance: for example, $\pi/6\in\Theta_0$, whereas $e^{i\pi/6}\notin\mathbb Q(\omega)$.

\begin{lemma}
\label{lem:ae-dichotomy}
For Lebesgue-almost every $\theta$, every direction $\vec e\in S^1$ has the following property: at least one of $\mu_{\vec e}$ and $\mu_{\mathbf R_\theta^T\vec e}$ satisfies exponential separation.
\end{lemma}

\begin{proof}
We first exclude angles at which two lattice vectors become too nearly parallel.
For non-zero $\vec\lambda,\vec\eta\in\boldsymbol\Lambda$, the determinant $\vec\lambda\wedge\mathbf R_\theta\vec\eta$ is a sine function of $\theta$ with amplitude $|\vec\lambda||\vec\eta|\ge1$.
Hence the set where $0<|\vec\lambda\wedge\mathbf R_\theta\vec\eta|<\delta$ has length at most $C\delta$.
At level $n$, there are at most $C16^n$ pairs with lattice coordinates bounded by $2^{n+1}$.
The same summation as in Lemma~\ref{lem:ae-exp-sep}, with threshold $16^{-n}n^{-2}$, shows that for almost every $\theta$, every such pair satisfies
\[
\vec\lambda\wedge\mathbf R_\theta\vec\eta=0
\quad\text{or}\quad
|\vec\lambda\wedge\mathbf R_\theta\vec\eta|\ge16^{-n}n^{-2}
\]
at all sufficiently large levels.
Fix such an angle outside $\Theta_0$, so the zero alternative is excluded.

Suppose that both projected systems fail exponential separation for some $\vec e$.
For every $A>0$ and all sufficiently large $n$, their cylinder differences then give non-zero lattice vectors $\vec\lambda_n,\vec\eta_n$, with coordinates bounded by $2^{n+1}$, such that
\[
|\langle\vec e,\vec\lambda_n\rangle|\le e^{-An},
\qquad
|\langle\vec e,\mathbf R_\theta\vec\eta_n\rangle|\le e^{-An}.
\]
Here multiplication of the cylinder differences by $2^n$ is absorbed by choosing a larger exponential rate before rescaling.
Both vectors lie within distance $e^{-An}$ of the line perpendicular to $\vec e$, and their lengths are at most $C2^n$.
Their determinant therefore satisfies
\[
16^{-n}n^{-2}
\le |\vec\lambda_n\wedge\mathbf R_\theta\vec\eta_n|
\le C2^ne^{-An}.
\]
Taking $A=4$ gives $n^{-2}\le C(32e^{-4})^n$, which is impossible for all large $n$.
The exceptional angle set was chosen before $\vec e$, so the conclusion holds simultaneously in every direction.
\end{proof}

\subsection{From measure dimension to intersection dimension}

For a compactly supported probability measure $\sigma$ on $\mathbb R^d$ and $q>1$, its $L^q$-dimension is
\[
D_q(\sigma):=\liminf_{n\to\infty}
\frac{-\log\sum_{Q\in\mathcal D_n}\sigma(Q)^q}{(q-1)n\log2},
\]
where $\mathcal D_n$ is the partition into half-open dyadic cubes of side $2^{-n}$.
This is \cite[Definition~1.3]{corso2024dynamical}, written with natural logarithms.
The ambient bound is $0\le D_q(\sigma)\le d$; see \cite[Section~1.2, inequality~(1.2)]{corso2024dynamical} after rescaling the compact support.
In particular, $D_q(\sigma)\le1$ on the line and $D_q(\sigma)\le2$ in the plane; the line bound is also stated in \cite[Section~1, p.~2, before~(1.1)]{rossi2018measures}.
We say that $\sigma$ is \emph{$q$-unsaturated on lines} if $D_q(\sigma)<D_q(\pi_*\sigma)+1$ for every orthogonal projection $\pi$ onto a $(d-1)$-dimensional subspace.
In dimension one, this condition is simply $D_q(\sigma)<1$.

We use the following specialisation of Corso--Shmerkin's result to systems whose common orthogonal part is the identity.

\begin{theorem}[Corso--Shmerkin, {\cite[Corollary~4.2]{corso2024dynamical}}]
\label{thm:corso-shmerkin}
Let $\sigma$ be a self-similar probability measure on $\mathbb R^d$ generated by finitely many maps $\vec x\mapsto\lambda\vec x+\vec a_i$, with $0<\lambda<1$ and positive weights $p_i$ summing to $1$.
If the system satisfies exponential separation and $\sigma$ is $q$-unsaturated on lines for some $q>1$, then
\[
D_q(\sigma)=\frac{\log\sum_i p_i^q}{(q-1)\log\lambda}.
\]
\end{theorem}

We apply this result with $\lambda=1/2$: in dimension one, to a gasket projection with three weights $1/3$; in dimension two, to $\nu_\theta$ with nine weights $1/9$.
Their symbolic dimensions are $\log3/\log2>1$ and $\log9/\log2>2$, respectively.
We apply the theorem by contradiction: a dimension drop would give unsaturation, forcing the dimension above that of the ambient space.

We also use the convolution inequality in Rossi--Shmerkin \cite[Section~1, equation~(1.2)]{rossi2018measures}: for compactly supported probability measures $\sigma,\tau$ on $\mathbb R$ and $q>1$,
\[
D_q(\sigma*\tau)\ge\max\{D_q(\sigma),D_q(\tau)\}.
\]
The cited equation gives the bound by the first factor; commutativity gives the bound by the second.

\begin{lemma}
\label{lem:line-lq}
Suppose $\theta$ satisfies the conclusion of Lemma~\ref{lem:ae-dichotomy}.
Then $D_q((\pi_{\vec e})_*\nu_\theta)=1$ for every $\vec e\in S^1$ and every $q>1$.
\end{lemma}

\begin{proof}
At least one of the two factors $\mu_{\vec e}$ and $\mu_{\mathbf R_\theta^T\vec e}$ satisfies exponential separation.
If its $L^q$-dimension were less than $1$, it would be $q$-unsaturated, and Theorem~\ref{thm:corso-shmerkin} would give dimension $\log3/\log2>1$.
Thus that factor has $L^q$-dimension $1$.
Reflection preserves $L^q$-dimension: reflected dyadic intervals and the original dyadic intervals have bounded overlap in both directions, so their $q$-moment sums differ by at most a factor depending only on $q$.
Applying the convolution inequality above to $\mu_{\vec e}$ and $(-\mathrm{id})_*\mu_{\mathbf R_\theta^T\vec e}$ gives $D_q((\pi_{\vec e})_*\nu_\theta)\ge1$; the line ambient bound gives the reverse inequality.
\end{proof}

\begin{proposition}
\label{prop:nu-lq}
Suppose $\theta$ satisfies the conclusions of Lemmas~\ref{lem:ae-exp-sep} and~\ref{lem:ae-dichotomy}.
Then $D_q(\nu_\theta)=2$ for every $q>1$.
\end{proposition}

\begin{proof}
Fix $q>1$ and suppose $D_q(\nu_\theta)<2$.
Lemma~\ref{lem:line-lq} gives $D_q((\pi_{\vec e})_*\nu_\theta)=1$ in every direction, so $\nu_\theta$ is $q$-unsaturated on lines.
Together with exponential separation, Theorem~\ref{thm:corso-shmerkin} then gives
\[
D_q(\nu_\theta)
=\frac{\log\bigl(9(1/9)^q\bigr)}{(q-1)\log(1/2)}
=\frac{\log9}{\log2}>2,
\]
a contradiction.
\end{proof}

Full $L^q$-dimension controls the mass near the origin.
Each intersecting pair of level-$n$ cells contributes probability $9^{-n}$ to a ball of radius comparable to $2^{-n}$, which turns this mass estimate into a covering bound.

\begin{theorem}
\label{thm:ae-upper}
For Lebesgue-almost every $\theta$,
\[
\overline{\dim}_{\mathrm B}(\mathcal S_\theta)\le2\dimhaus(S)-2.
\]
\end{theorem}

\begin{proof}
Fix an angle satisfying the two separation lemmas.
Let $Q_n(\theta)$ be the number of ordered pairs of length-$n$ addresses whose gasket cells intersect after rotation.
Choose two independent uniform symbolic addresses.
Each prescribed pair of prefixes has probability $9^{-n}$.
If its two cells intersect, every pair of points in them is at distance at most $2\operatorname{diam}(S)2^{-n}$.
The prefix events are disjoint in the symbolic probability space, even when geometric cells share boundary points; consequently,
\[
9^{-n}Q_n(\theta)
\le\nu_\theta\bigl(\overline B(\vec{0},2\operatorname{diam}(S)2^{-n})\bigr).
\]

Fix $q>1$ and $\varepsilon>0$.
By Proposition~\ref{prop:nu-lq} and the definition of $L^q$-dimension,
\[
\sup_{Q\in\mathcal D_n}\nu_\theta(Q)
\le\left(\sum_{Q\in\mathcal D_n}\nu_\theta(Q)^q\right)^{1/q}
\le C_{\theta,q,\varepsilon}\,2^{-n(2(1-1/q)-\varepsilon)}.
\]
The ball above meets only a bounded number of level-$n$ dyadic squares.
The intersection is covered by at most $Q_n(\theta)$ blue level-$n$ cells, each of diameter $\operatorname{diam}(S)2^{-n}$.
Combining these covering observations with the two estimates gives
\[
\overline{\dim}_{\mathrm B}(\mathcal S_\theta)
\le\frac{\log9}{\log2}-2+\frac2q+\varepsilon.
\]
Since this holds for every $q>1$ and $\varepsilon>0$, the required bound follows.
\end{proof}

The bound is $\log(9/4)/\log2\approx1.169925$, obtained from $\dimhaus(S)=\log3/\log2$.

\begin{remark}
\label{rem:ae-upper-resonance}
The proof excludes the directional resonances $\Theta_0$, as well as the null sets in the separation lemmas.
At a directional resonance the determinant of a non-zero lattice pair can vanish, so the lower bound used in Lemma~\ref{lem:ae-dichotomy} is unavailable.
The angle $\pi/6$ is one such example, discussed in Section~\ref{sec:conjecture}.
Since $\Theta_0$ is countable, its exclusion does not affect the almost-everywhere statement.
\end{remark}

\section{The lower bound and numerical evidence}
\label{sec:ae-lower}

Theorem~\ref{thm:ae-upper} gives the expected upper bound for almost every angle.
The matching Hausdorff lower bound remains open.
We explain the obstacles, then describe the numerical experiments that motivate the conjecture.

\subsection{Why the lower bound is difficult}

The first difficulty is that the relative displacement is fixed.
Our intersection is $S\cap\mathbf R_\theta S$, so the two centres coincide at every angle.
Mattila's intersection theorem \cite[Theorem~4.1]{Mattila2021} already gives the expected dimension for a positive-measure set of relative displacements at almost every angle.
Indeed, the gasket is Ahlfors regular and $\frac32\dimhaus(S)=\frac{3\log3}{2\log2}>2$, so the hypotheses of that theorem hold with both sets equal to $S$.
Consequently, for almost every $\theta$,
\[
\dimhaus\bigl(S\cap(\mathbf R_\theta S-\vec c)\bigr)=2\dimhaus(S)-2
\]
for a set of $\vec c\in\mathbb R^2$ of positive Lebesgue measure.
The theorem does not determine whether $\vec c=\vec{0}$ belongs to this set.
An argument that averages over translations therefore needs an additional estimate at zero relative displacement to address our problem.

The difference measure $\nu_\theta$ makes this difficulty visible.
The proof in Section~\ref{sec:ae-upper} bounds $\nu_\theta(B(\vec{0},r))$ from above and uses that bound to control how many pairs of small pieces can meet.
A lower bound would require information in the opposite direction, which full $L^q$-dimension alone does not supply at a prescribed point.
Moreover, local dimension records only the logarithmic rate of decay of this mass: local dimension $2$ does not imply a positive lower density.
Neither statement by itself identifies zero as a translation for which the intersection has the expected Hausdorff dimension.

There is a further difficulty in passing from counts to Hausdorff dimension.
The number $Q_n(\theta)$ of intersecting pairs at level $n$ records how much of the intersection is visible at scale $2^{-n}$.
An exponential growth rate for these counts would give information about box dimension, but does not by itself control how the surviving pieces are distributed among their ancestors.
The mass distribution principle and Frostman's lemma make the required control across scales precise \cite[Mass distribution principle~4.2 and Corollary~4.12]{Falconer2003}.
To prove the Hausdorff lower bound, it is enough to construct, for each $0<t<2\dimhaus(S)-2$, a probability measure $\eta$ supported on $\mathcal S_\theta$ satisfying $\eta(B(\vec x,r))\le C_t r^t$ for every $\vec x\in\mathbb R^2$ and $r>0$; conversely, the conjectured lower bound would ensure the existence of such measures.
At resonant angles the finite graph provides such structure through its Perron--Frobenius eigenvectors.
At non-resonant angles the reachable relative displacements are infinite, and the recursion alone gives no corresponding uniform mass estimate.

\subsection{Numerical experiments and the conjecture}
\label{sec:conjecture}

The two classes of angles lead to different computations. At resonant angles, the dimension is given by a theorem, and we illustrate its dependence on the rotation angle. At non-resonant angles, finite-depth intersection counts test the conjectured dimension $2\dimhaus(S)-2$.

\paragraph{Resonant angles: dimensions from spectral radii.}
For a resonant angle, construct the finite matrix $\mathbf A_\theta$ from Definition~\ref{def:relative-displacement-matrix}.
Theorem~\ref{thm:commensurable-gives-finite-type} gives the exact formula $\dimhaus(\mathcal S_\theta)=\dimbox(\mathcal S_\theta)=\log_2\rho(\mathbf A_\theta)$.
Thus the dimension follows by evaluating the spectral radius, without extrapolation from finite depths.
The decimal values below use numerical evaluations of the spectral radius.
The sample is parametrised by $e^{i\theta}=\overline{(m+n\omega)}/(m+n\omega)$, with size parameter $N:=m^2-mn+n^2$.
The stored data contain $218$ points over one period $[0,2\pi/3]$, including reflected copies, with $7\le N\le597$.
Table~\ref{tab:commensurable-dimensions} gives selected values, together with the known symmetry cases.
The blue points in Figure~\ref{fig:commensurable-dimensions} show how these dimensions vary with the rotation angle; apart from the symmetry cases, the sampled dimensions lie between $1.14468$ and $1.20733$.

\paragraph{Non-resonant angles: testing the conjecture by intersection counts.}
For $n\ge0$, define the intersection count
\[
Q_n(\theta):=\#\bigl\{(\mathbf i,\mathbf j)\in\mathcal I^n\times\mathcal I^n:
F_{\mathbf i}(S)\cap\mathbf R_\theta F_{\mathbf j}(S)\neq\varnothing\bigr\}.
\]
This counts ordered address pairs whose level-$n$ cells intersect after rotation; here $F_{\mathbf i}:=F_{i_1}\circ\cdots\circ F_{i_n}$, with the empty word giving the identity.
Thus we count address pairs, rather than distinct relative displacements or distinct intersection sets.
These counts are comparable to covering numbers.
Let $N_r(\mathcal S_\theta)$ be the least number of closed balls of radius $r$ needed to cover $\mathcal S_\theta$.
Then, with an absolute constant $C$ independent of $n$ and $\theta$,
\[
N_{2^{-n}}(\mathcal S_\theta)\le Q_n(\theta)\le C N_{2^{-n}}(\mathcal S_\theta).
\]
Indeed, the intersection of each pair is contained in its blue cell, which lies in a ball of radius $2^{-n}$ about its centre.
Conversely, distinct blue level-$n$ cell centres are distinct points of $2^{-n}\boldsymbol\Lambda$, hence are separated by at least $2^{-n}$; the same holds for the red centres after rotation.
A ball of radius $2^{-n}$ meeting a cell has its centre within $2^{1-n}$ of that cell's centre, so a planar packing bound limits the number of blue and red cells it can meet by an absolute constant.
For each intersecting pair, choose a point in the intersection and a covering ball containing it; each ball is assigned only a bounded number of pairs.
This proves the reverse inequality.
Thus the lower and upper box dimensions are respectively the limit inferior and limit superior of $\log_2 Q_n(\theta)/n$; finite-depth slopes estimate this growth rate.

For a non-resonant angle, we start from $\vec c_0:=\vec{0}$ and iterate
\[
\vec c_{k+1}:=2\vec c_k+\vec p_{i_{k+1}}-\mathbf R_\theta\vec p_{j_{k+1}},
\qquad (i_{k+1},j_{k+1})\in\mathcal I^2.
\]
By Lemma~\ref{lem:filling}, a child pair intersects precisely when its relative displacement belongs to the polygon $\mathbf R_\theta T-T$.
We retain these children and count the surviving ordered pairs of address words to obtain $Q_n(\theta)$.
Different pairs are counted separately even when they have the same relative displacement.
We compute through depth $20$ and examine $(4/9)^nQ_n(\theta)$, the normalisation corresponding to the predicted growth rate $9/4$.
The implementation uses double precision; changing the tolerance in the polygon membership test from $10^{-10}$ to $10^{-8}$ leaves all reported counts unchanged.
The counting and plotting scripts, together with the counts at every depth, are available at \url{https://github.com/Nero-17/LEAN-Formalisation-Moire-Pattern-Sierpinski-Gasket/tree/main/numerics}.
The blue-point data, red-point estimates and all non-resonant counts are also available in a compact JSON file in the Lean repository: \url{https://github.com/Nero-17/LEAN-Formalisation-Moire-Pattern-Sierpinski-Gasket/blob/main/data/dimension-results.json}.

Figure~\ref{fig:survivor-counts} shows the normalised counts at four non-resonant angles.
We fit $\log_2Q_n(\theta)$ against $n$ by least squares with an intercept, using the slope as a finite-scale growth exponent.
Table~\ref{tab:survivor-counts} compares the results over different depth intervals.
Using $n=16,\ldots,20$ for all four angles gives slopes within $0.0008$ of $2\dimhaus(S)-2\approx1.169925$, each closer to this value than the corresponding fit over $n=8,\ldots,12$.
Figure~\ref{fig:commensurable-dimensions} plots these deeper estimates in red, with the conjectured value marked by a red dashed line.
The sample includes $\pi/6$, where the directional resonance discussed in Remark~\ref{rem:ae-upper-resonance} prevents direct application of the argument in Section~\ref{sec:ae-upper}.

The fits still depend on the depth interval: at $\arctan(\sqrt3/2)$, for example, fitting over $n=12,\ldots,20$ gives a slope of approximately $1.165102$.
The differences above are therefore not rigorous error bounds for the dimension, and finite-depth counts do not establish a Hausdorff lower bound.
Nevertheless, the closer agreement with the predicted growth rate at greater depths provides numerical support for the dimension conjecture below.

\begin{figure}[H]
\centering
\begin{tikzpicture}
\begin{axis}[
    width=0.83\linewidth, height=4.5cm,
    xmin=0, xmax=20, ymin=0.8, ymax=5.6,
    xtick={0,4,8,12,16,20}, ytick={1,2,3,4,5},
    xlabel={depth $n$}, ylabel={$(4/9)^nQ_n(\theta)$},
    grid=major, grid style={gray!18},
    tick label style={font=\small}, label style={font=\small},
    legend style={at={(0.5,-0.42)}, anchor=north, legend columns=4,
                  font=\small, draw=none,
                  /tikz/every even column/.append style={column sep=7pt}},
]
\addplot[blue!75!black, thick, mark=*, mark size=1.4pt]
    table[x=depth,y=pi_four]{numerics/survivor_plot.dat};
\addlegendentry{$\pi/4$}
\addplot[crimson, thick, mark=square*, mark size=1.4pt]
    table[x=depth,y=pi_six]{numerics/survivor_plot.dat};
\addlegendentry{$\pi/6$}
\addplot[black, thick, mark=triangle*, mark size=1.7pt]
    table[x=depth,y=pi_twelve]{numerics/survivor_plot.dat};
\addlegendentry{$\pi/12$}
\addplot[green!45!black, thick, mark=diamond*, mark size=1.7pt]
    table[x=depth,y=atan_sqrt3_half]{numerics/survivor_plot.dat};
\addlegendentry{$\arctan(\sqrt3/2)$}
\end{axis}
\end{tikzpicture}
\caption{Normalised intersection counts at four non-resonant angles.
The normalisation corresponds to the dimension $\log_2(9/4)=2\dimhaus(S)-2$.}
\label{fig:survivor-counts}
\end{figure}
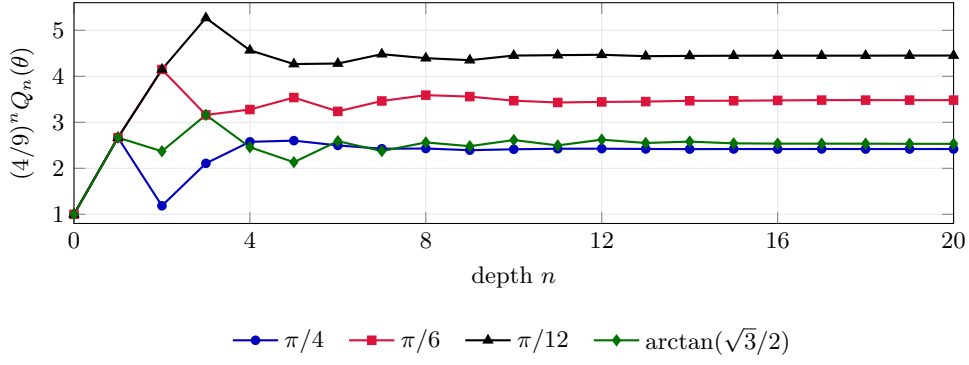

\begin{table}[H]
\centering
\begin{minipage}[t]{0.41\textwidth}
\vspace{0pt}
\footnotesize
\setlength{\tabcolsep}{2pt}
\centering
\begin{tabular}{c|rrr}
\hline
$\theta$ & $8$--$12$ & $16$--$20$ & $18$--$20$\\
\hline
$\pi/4$ & $1.171370$ & $1.170000$ & $1.170139$\\
$\pi/6$ & $1.152566$ & $1.170392$ & $1.169893$\\
$\pi/12$ & $1.178328$ & $1.170130$ & $1.170307$\\
$\arctan(\sqrt3/2)$ & $1.177542$ & $1.169157$ & $1.168877$\\
\hline
\end{tabular}
\caption{Least-squares slopes of $\log_2Q_n(\theta)$ against $n$ over the indicated depth intervals.}
\label{tab:survivor-counts}
\end{minipage}\hfill
\begin{minipage}[t]{0.57\textwidth}
\vspace{0pt}
\footnotesize
\setlength{\tabcolsep}{3pt}
\centering
\begin{tabular}{ccccc}
\hline
$(m,n)$ & $N$ & $\theta$(rad) & $\log_2\rho(\mathbf A_\theta)$ & \shortstack{difference from\\ $2\dimhaus(S)-2$}\\
\hline
$(1,0)$ & $1$ & $0$ & $1.5849625$ & $+0.4150375$\\
$(1,-1)$ & $3$ & $\pi/3$ & $1.2924813$ & $+0.1225563$\\
$(3,1)$ & $7$ & $0.6670$ & $1.1702210$ & $+0.0002960$\\
$(4,1)$ & $13$ & $0.4851$ & $1.1981591$ & $+0.0282341$\\
$(5,1)$ & $21$ & $0.3803$ & $1.2073297$ & $+0.0374047$\\
$(5,2)$ & $19$ & $0.8173$ & $1.1680358$ & $-0.0018892$\\
$(7,1)$ & $43$ & $0.2650$ & $1.1694128$ & $-0.0005122$\\
\hline
\end{tabular}
\caption{Selected values from the resonant-angle computation, with the symmetry cases added.
}
\label{tab:commensurable-dimensions}
\end{minipage}
\end{table}

\begin{figure}[H]
\centering
\includegraphics[width=0.7\linewidth]{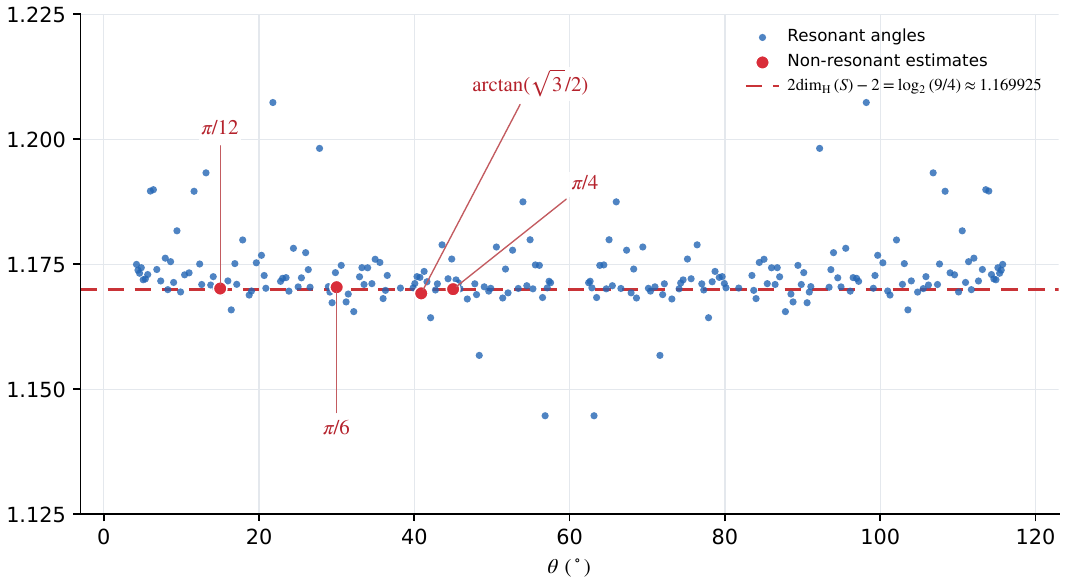}
\caption{Dimensions and finite-scale estimates.}
\label{fig:commensurable-dimensions}
\end{figure}

The spectral-radius calculation displays the angular variation of the resonant dimensions. The non-resonant counting experiments support the following conjecture, whose lower bound remains open.

\begin{conjecture}
\label{conj:non-resonant-dimension}
Let $\theta\in[0,\pi/3]$ with $e^{i\theta}\notin\mathbb Q(\omega)$.
Then
\[
\dimhaus(\mathcal S_\theta)=\dimbox(\mathcal S_\theta)=2\dimhaus(S)-2.
\]
\end{conjecture}

\section*{Acknowledgement}
The author would like to thank Ruiqi Wu and Alex M. Ganose from the Department of Chemistry at Imperial College London for inspiring this paper, whose initial motivation came from twisted bilayer graphene.
The author also thanks Yaxin Zheng for her help.

This work was supported by the Additional Funding Programme for Mathematical Sciences, delivered by EPSRC (EP/V521917/1) and the Heilbronn Institute for Mathematical Research, and also by the EPSRC Centre for Doctoral Training in Mathematics of Random Systems: Analysis, Modelling and Simulation (EP/S023925/1).

\bibliographystyle{plain}
\bibliography{references}

\end{document}